\documentclass{amsart} 
\usepackage[top=1.5in, bottom=1in, left=0.9in, right=0.9in]{geometry}

\usepackage[utf8]{inputenc}
\usepackage[all]{xy}
\usepackage{color}
\usepackage{hyperref}
\usepackage{tikz-cd}

\usepackage{amsmath, mathtools}
\hypersetup{colorlinks=true}

\usepackage{amssymb}
\usepackage{amsfonts}
\usepackage{graphicx}
\usepackage{tikz}
\usepackage{mathrsfs}
\usepackage{amscd}
\usepackage[all]{xy}
\usepackage{amsthm}
\usepackage{cleveref}

\usepackage{verbatim}
\usepackage{amscd}
\usepackage{mathtools}
\usepackage{comment}
\usepackage{arydshln}
\usepackage{mathtools}
\usepackage{latexsym, amsthm, amsfonts, mathrsfs, amscd,enumerate,multicol}
\newtheorem{thm}{Theorem}[section]
\newtheorem{lem}[thm]{Lemma}
\newtheorem{prop}[thm]{Proposition}

\newtheorem{cor}[thm]{Corollary}
\theoremstyle{definition}
\newtheorem{dfn}[thm]{Definition}
\newtheorem{defn}[thm]{Definition}
\newtheorem{ques}[thm]{Question}
\newtheorem{claim}{Claim}
\newtheorem{setup}[thm]{Setup}
 
\newtheorem{rem}[thm]{Remark}

\newtheorem{ex}[thm]{Example}
\newtheorem{fact}[thm]{Fact} 

\newtheorem{setting}[thm]{Setting}
\theoremstyle{remark}

\numberwithin{equation}{thm}

\def\depth{\operatorname{depth}}
 
\def\tr{\operatorname{tr}}     
\def\Ext{\operatorname{Ext}}
 
\def\hom{\operatorname{Hom}} 
\def\Hom{\operatorname{Hom}}  
 
\def\m{\mathfrak{m}}

\def \Ass {\operatorname{Ass}}

\def \Spec{\operatorname{Spec}}

\def\p{\mathfrak{p}}

\def \Max{\operatorname{Max}}
\def \Min {\operatorname{Min}}

\def \Ker {\operatorname{Ker}}
\def \Im {\operatorname{Im}}

\newcommand{\rme}{\mathrm{e}}

\newcommand{\rmr}{\mathrm{r}}

\newcommand{\rmQ}{\mathrm{Q}}

\newcommand{\fkm}{\mathfrak{m}}
\newcommand{\fkn}{\mathfrak{n}}

\newcommand{\fkp}{\mathfrak{p}}
\newcommand{\fkq}{\mathfrak{q}}

\def\ol{\overline}

\subjclass[2020]{13B22
, 13B30
, 13H10}
\thanks{Souvik Dey was partly supported by the Charles University Research Center program No.UNCE/SCI/022 and a grant GA CR 23-05148S from the Czech Science Foundation. 
Shinya Kumashiro was supported by JSPS KAKENHI Grant Number JP24K16909.
}

\begin{document} 

\title{On partial trace ideals} 

\author{Souvik Dey}  
\address{Souvik Dey: Department of Mathematical Sciences, University of Arkansas, 850 West Dickson Street Fayetteville, Arkansas 72701 United States} 
\email{souvikd@uark.edu} 
\author{Shinya Kumashiro}
\address{Shinya Kumashiro: Department of Mathematics, Osaka Institute of Technology, 5-16-1 Omiya, asahi-ku, Osaka, 535-8585, Japan}
\email{shinya.kumashiro@oit.ac.jp}

\begin{abstract}
We investigate the notion of partial trace ideals, recently introduced by Maitra. We first establish several properties of partial trace ideals and give affirmative answers to questions posed by Maitra. We then study the invariant defined by the partial trace ideal of the canonical module, and obtain an upper bound that recovers one direction of a result of Kobayashi. Moreover, in the case of numerical semigroup rings generated by three elements, we provide an explicit formula for this invariant.
\end{abstract}

\maketitle  
   
\section{Introduction}

Let $R$ be a Noetherian ring and $M$ an $R$-module. The trace ideal of $M$, the sum of all images of $\Hom_R(M, R)$,  has long been known as a useful tool in commutative algebra (for example, see \cite{Ba, Ding, Vas}). In recent years, motivated by Lindo's work \cite{Lin}, there has been active research on the ubiquity of trace ideals and on the trace ideal of the canonical module. 

On the other hand, very recently, Maitra \cite{sm} introduced the notion of a partial trace ideal of $M$. For an $R$-module $M$, let  
\[
h(M):=\inf \{\ell_R(R/\Im f) \mid f\in \hom_R(M,R)\}. 
\]
With this notation, a partial trace ideal of $M$ is an image of $f\in \Hom_R(M, R)$ such that $\ell_R(R/\Im f) =h(M)$. 
Maitra has used the notion of partial trace ideals to obtain new progress on the Berger conjecture \cite{sm}, and has also given a characterization of the almost Gorenstein property \cite{Mai2}. However, several fundamental properties of partial trace ideals remain open. In this paper, we list below the problems that we address. 
Let $\ol{R}$ denote the integral closure of $R$. We consider the colon in the total quotient ring $\rmQ(R)$.
\begin{ques}
\begin{itemize}
\item When is $h(M)$ finite?  
\item (\cite[Question 5.1]{Mai2}): For a given $R$-module $M$, how many partial trace ideals of $M$ exist? 
\item (\cite[after Theorem B, after Theorem 2.5, and Conclusion (2)]{sm}): Let $(R, \fkm)$ be a Cohen-Macaulay local ring of dimension one.  For an $\fkm$-primary ideal $I$, $I$ is a partial trace ideal for some $R$-module if and only if $R:I\subseteq \ol{R}$?
\end{itemize}
\end{ques}

The first purpose of this paper is to answer the questions stated above, and we establish the following.

\begin{thm}{\rm (Proposition \ref{p21} and Theorem \ref{1dim})}\label{aaa11}
Let $\displaystyle S=R\setminus \bigcup_{\fkp\in \Spec (R)\setminus \Max (R)}\fkp$ be a multiplicatively closed subset of $R$. Consider the following conditions.
\begin{enumerate}[\rm(1)] 
\item $h(M)<\infty$.
\item $S^{-1}R$ is a direct summand of $S^{-1}M$. 
\item $\ell_R(R/\tr_R(M))<\infty$.
\item $M_\fkp$ has a free summand for all $\fkp\in \Spec (R)\setminus \Max (R)$.
\end{enumerate}
Then, {\rm (1)$\Leftrightarrow$(2)$\Rightarrow$(3)$\Leftrightarrow$(4)} hold. {\rm (3)$\Rightarrow$(2)} holds if $R$ is local and $\dim R=1$. 
\end{thm}

\begin{thm}{\rm (Theorem \ref{t26} and Corollary \ref{corcor})}\label{aaa12}
Let $(R, \fkm)$ be a Cohen-Macaulay local ring of dimension one, and let $I$ be an $\fkm$-primary ideal of $R$.

\begin{enumerate}[\rm(1)] 
\item Suppose that all $\fkm$-primary ideals have a principal reduction (for example, the residue field $R/\fkm$ is infinite or $\ol{R}$ is local with the same residue field as that of $R$). Then, $I$ is a partial trace ideal for some $R$-module if and only if $R:I\subseteq \ol{R}$.

\item Suppose that $\ol{R}$ is a discrete valuation ring with the maximal ideal $\fkn$ and $R/\fkm\cong \ol{R}/\fkn$. Suppose that $J$ is a partial trace ideal of $I$. 
Then, 
\begin{align*}
\{\text{partial trace ideals of $I$}\} = \{\alpha J \mid \alpha\in R:J, v(\alpha) =0\},
\end{align*}
\end{enumerate}
where $v(\alpha)$ denotes the normalized valuation associated to $\ol{R}$ (see before Remark \ref{tuesday}).
\end{thm}

The second aim of this paper is to analyze the invariant $h(\omega_R)$, where $\omega_R$ denotes the canonical module of $R$. In what follows, we assume that $(R, \fkm)$ is a Cohen–Macaulay local ring admitting a canonical module $\omega_R$. A simple observation shows that $R$ is Gorenstein if and only if $h(\omega_R)=0$. Moreover, Artin local rings with $h(\omega_R)=1$ are known as Teter rings, and their structure was clarified by several papers (see, for example, \cite{AAM, HV}). In addition, Kobayashi \cite{Ko} has obtained the following result for $R$ with $\dim R=1$ and $h(\omega_R)=2$.

\begin{fact}(\cite[Theorem 1.4]{Ko}) Suppose that $R$ is a Cohen-Macaulay local ring of dimension one having the canonical module $\omega_R$. Then, $h(\omega_R)\le 2$ if and only if $bg(R)\le 1$, where 
\[
bg(R):=\inf \{\ell_S(R/S) \mid S\subseteq R \text{ is a module finite}, \rmQ(S)=\rmQ(R), \text{ and $S$ is Gorenstein}\}.
\]
\end{fact}

We give another characterization of the condition $h(\omega_R)= 2$ that is different from the above (Theorem \ref{h=2}). Moreover, we recover the ``if'' part of the above theorem as follows.

\begin{thm} {\rm (Theorem \ref{tbg})}\label{aaa13}
Let $(R,\m)$ be a one-dimensional Cohen-Macaulay local ring having a canonical module $\omega_R$. If $R$ is generically Gorenstein, then $h(\omega_R)\leq 2bg(R)$. 
\end{thm}

We further provide a formula for computing $h(\omega_R)$ in the case of three-generated numerical semigroup rings (Theorem \ref{thm4.4}). 

We describe the organization of this paper. In Section \ref{sec2}, we explore properties of partial trace ideals and prove Theorems \ref{aaa11} and \ref{aaa12}. In Section \ref{sec3}, we study $h(\omega_R)$ and prove Theorem \ref{aaa13}. In Section \ref{sec4}, we provide a formula for $h(\omega_R)$ in the case of three-generated numerical semigroup rings. Examples are given throughout the paper.

\begin{setup}
In what follows, all rings are commutative Noetherian rings with identity and all modules are assumed to be finitely generated. Let $R$ be a ring. 
Let $\rmQ(R)$ and $\ol{R}$ denote the {\it total quotient ring} of $R$ and the {\it integral closure} of $R$, respectively. 
An $R$-module $X$ is called a {\it fractional ideal} if $X\subseteq \rmQ(R)$ and $X$ contains a non-zerodivisor of $R$. 
Let $X, Y$ be fractional ideals of $R$. We denote by $Y:X$ the colon fractional ideal considered in $\rmQ(R)$. It is well-known that there is an isomorphism
\[
Y:X \cong \Hom_R(X, Y); \ \alpha \mapsto \cdot \alpha,
\]
where $\cdot \alpha$ means the homomorphism $X\to Y$ defined by multiplying $\alpha\in Y:X$ (\cite{HKun}), and we will freely use this fact. When we consider the colon in $R$, we denote it by $Y:_R X$. We say that an ideal $I$ is {\it regular} if $I$ contains a non-zerodivisor of $R$. 

If $R$ is local, then $\fkm$ denotes the unique maximal ideal of $R$. 
For an $R$-module $M$ over a local ring $R$, we denote by $\rme(M), \mu_R(M), \ell_R(M)$ the {\it multiplicity}, the {\it number of minimal generators}, the {\it length}, respectively. If $R$ is a Cohen-Macaulay local ring, $\rmr(R)$ denotes the {\it Cohen-Macaulay type} of $R$. 

\end{setup}

\section{Trace ideals and partial trace ideals}\label{sec2}

Throughout this section, let $R$ be a Noetherian ring and $M$ an $R$-module.

\subsection{Trace ideals}
In this subsection we summarize basic properties of trace ideals. 

\begin{dfn}
For an $R$-module $M$, the ideal 
\[
\tr_R(M)=\sum_{f\in \hom_R(M, R)} \mathrm{Im} f
\]
is called the {\it trace ideal of $M$}. An ideal $I$ is called a {\it trace ideal} if $I$ is a trace ideal of some $R$-module.
\end{dfn}

\begin{fact}\label{r12}
The following hold true. 
\begin{enumerate}[\rm(1)] 
\item $\tr_R(M) = \Im \varphi$, where $\varphi$ is the evaluation map 
\[
\varphi:\Hom_R(M, R) \otimes_R M \to R; f\otimes x \mapsto f(x) \quad \text{for $f\in \Hom_R(M, R)$ and $x\in M$. }
\] 
When $M$ is a fractional ideal $I$, the map $\varphi$ can be identified with $(R:I)\otimes_R I \to R; f\otimes x \mapsto fx$; hence, $\tr_R(I)=(R:I)I$. 
\item (\cite[Proposition 2.8(viii)]{Lin}): $\tr_R(M)_\fkp = \tr_{R_\fkp}(M_\fkp)$ for all $\fkp\in \Spec (R)$. 
\item (cf. \cite[Proposition 2.8(iii)]{Lin}): $\tr_R(M)=R$ if and only if $M^{\oplus n}$ has a free summand for some $n>0$. If $R$ is local, $n=1$ suffices. 
\end{enumerate}
\end{fact}

The following is a characterization of trace ideals. 

\begin{fact}{\rm (\cite[Corollary 2.2]{GIK})}\label{f22}
Let $I$ be an ideal of $R$ containing a non-zerodivisor of $R$. Then the following are equivalent. 
\begin{enumerate}[\rm(1)] 
\item $I$ is a trace ideal. 
\item $I$ is the trace ideal of $I$, that is, $I=(R:I)I$.
\item $R:I=I:I$.
\end{enumerate}
\end{fact}


\subsection{Partial trace ideals}


\begin{dfn}(\cite[Definition 2.1]{sm})
For an $R$-module $M$, we define an invariant  
\[
h(M):=\inf \{\ell_R(R/\Im f) \mid f\in \hom_R(M,R)\}
\]
and call it the {\it $h$-invariant} of $M$. We say that an ideal $J$ is {\it a partial trace ideal of $M$} if $J=\Im f$ for some $f\in \hom_R(M,R)$ and $\ell_R(R/J)=h(M)$. An ideal $J$ is called a {\it partial trace ideal} if $J$ is a partial trace ideal of some $R$-module $M$. 
\end{dfn}

\begin{rem}
\begin{enumerate}[\rm(1)] 
\item For an ideal $I$, $I$ is a partial trace ideal if and only if $I$ is a partial trace ideal of $I$.
\item Any ideal of an Artinian ring is a partial trace ideal of itself. 
\item As we will see later, a partial trace ideal of a module is not uniquely determined by the module. Indeed, let $R=k[[t^3, t^4, t^5]]$, where $k$ is a field and $k[[t]]$ is the formal power series ring over $k$, and $M=(t^4, t^5)$. For $a\in k$, set $I_a=(t^3+at^4, t^4+at^5)$. 
Then, $I_a$ are partial trace ideals of $M$ for all $a\in k$. On the other hand, $I_a=I_b$ if and only if $a=b$. 
\end{enumerate}
\end{rem}

\begin{proof}
(1): See \cite[Remark 2.5]{sm}.

(2): Let $R$ be an Artinian ring and $I$ an ideal of $R$. Let $f\in \Hom_R(I, R)$. Since $I \xrightarrow{f} \Im f \to 0$, we get 
\[
\ell_R(I) \ge \ell_R(\Im f) =\ell_R(R) - \ell_R(R/\Im f). 
\]
Hence, $\ell_R(R/\Im f) \ge \ell_R(R)-\ell_R(I) = \ell_R(R/I)$ for all $f\in \Hom_R(I, R)$. It follows that $h(I)\ge \ell_R(R/I)$. On the other hand, by considering the embedding $\iota: I\to R$, we get $\ell_R(R/\Im \iota) = \ell_R(R/I)$; hence, $I=\Im \iota$ is a partial trace ideal of $I$.

(3): By considering the homomorphisms $\iota: I\to R; x\mapsto x$ and $f:I\to R;x\mapsto tx$, we get $\tr_R(I) \supseteq \Im \iota + \Im f =\fkm$. Since $R$ is a domain and $I$ is not principal, $\tr_R(I)\ne R$ (Fact \ref{r12}(3)). Thus, $\tr_R(I) = \fkm$. By Remark \ref{tuesday}, $(t^3+at^4)$ is a reduction of $\fkm=\tr_R(I)$. Noting $M \cong I_a$ via the multiplication $t^{3}+at^4$, it follows that $I_a$ is a partial trace ideal of $M$ by Theorem \ref{t26}. 

Assume $I_a=I_b$. Then, $t^3+at^4\in I_b = k(t^3+bt^4) + k(t^4+bt^5) +\sum_{i\ge 6}k t^i$. Hence, there exist $c_i\in k$ such that 
\[
t^3+at^4 = c_3(t^3+bt^4) + c_4(t^4+bt^5) +\sum_{i\ge 6}c_i t^i.
\]
By comparing coefficients, we get $1=c_3$, $a=bc_3+c_4$, and $0=c_4$, i.e., $a=b$.
\end{proof}

As we shall see from the way partial trace ideals are defined, we need to consider the condition of when $h(M)<\infty$. We start with the following. 

\begin{prop}\label{p21}
Let $\displaystyle S=R\setminus \bigcup_{\fkp\in \Spec (R)\setminus \Max (R)}\fkp$ be a multiplicatively closed subset of $R$. 
Consider the following conditions.
\begin{enumerate}[\rm(1)] 
\item $h(M)<\infty$.
\item $S^{-1}R$ is a direct summand of $S^{-1}M$. 
\item $\ell_R(R/\tr_R(M))<\infty$.
\item $M_\fkp$ has a free summand for all $\fkp\in \Spec (R)\setminus \Max (R)$.
\end{enumerate}
Then, {\rm (1)$\Leftrightarrow$(2)$\Rightarrow$(3)$\Leftrightarrow$(4)} hold. 
\end{prop}

\begin{proof}
(1)$\Rightarrow$(2): Choose a homomorphism $f\in \Hom_R(M, R)$ such that $\ell_R(R/\Im f) < \infty$. Since $S^{-1}(\Im f)=S^{-1}R$, we get a surjection $S^{-1}f : S^{-1}M\to S^{-1}R$. This homomorphism splits. 

(2)$\Rightarrow$(1): There exists a surjective homomorphism $\varphi \in \Hom_{S^{-1}R} (S^{-1}M, S^{-1}R)$. On the other hand, since $M$ is finitely presented, $\Hom_{S^{-1}R}(S^{-1}M, S^{-1}R) \cong S^{-1}(\Hom_{R}(M, R))$. Thus, there exist $f\in \Hom_R(M, R)$ and $s\in S$ such that $f/s = \varphi$. Since $s/1$ is a unit of $S^{-1}R$, it follows that
\[
\Im(f/1) =(1/s) {\cdot} \Im(f/1) = \Im (f/s) = \Im \varphi =S^{-1} R.
\]
Since $(\Im f)_\fkm \cong (S^{-1}(\Im f))_\fkm \cong (S^{-1}R)_\fkm =R_\fkm$ for each $\fkm \in \Max R$, $\ell_R(R/\Im f)<\infty$; hence, $h(M)<\infty$.

(1)$\Rightarrow$(3): This is easy since $J\subseteq \tr_R(M)$ for any partial trace ideal $J$ of $M$. 

(3)$\Leftrightarrow$(4): This follows by Fact \ref{r12}(2) and (3).
\end{proof}

The implication (3)$\Rightarrow$(2) of Proposition \ref{p21} does not hold in general (Example \ref{ex28}). We prepare a lemma.

\begin{lem}\label{addl29}
Let $(R, \fkm)$ be a Cohen-Macaulay local ring. If there exists a maximal Cohen-Macaulay $R$-module $M$ of rank one such that $h(M)<\infty$, then $\dim R \le 1$.
\end{lem}

\begin{proof}
Let $J=\Im f$ be a partial trace ideal of $M$, where $f\in \Hom_R(M, R)$. Then, $\ell_R(R/J)=h(M)<\infty$ and we have an exact sequence
\[
0 \to \Ker f \to M \xrightarrow{f} R \to R/J \to 0.
\]
Suppose that $\Ker f\ne 0$ and choose $\fkq\in \Ass(\Ker f)$. Then, $\fkq\in \Ass_R (M)\subseteq \Ass (R)$ since $M$ is maximal Cohen-Macaulay. By localizing the above exact sequence at $\fkq$, we get a splitting exact sequence $0 \to (\Ker f)_\fkq \to M_\fkq \to R_\fkq \to 0$. Since $M$ is of rank one, $(\Ker f)_\fkq=0$. This is a contradiction since $\fkq \in \Ass(\Ker f)$. It follows that $\Ker f=0$. Hence, $f$ is injective and we get $0 \to M \xrightarrow{f} R \to R/J \to 0$. Since $M$ and $R$ are maximal Cohen-Macaulay and $R/J$ is of finite, by the depth lemma, $\dim R \le 1$. 
\end{proof}

\begin{rem} 
If the assumption on the rank of $M$ is dropped from Lemma \ref{addl29}, the conclusion fails. For example, let $R=k[[x^3, x^2y, xy^2, y^3]]$ and $M=(x,y)\oplus (x^2,xy, y^2)$. Then, $R$ is a Cohen-Macaulay local ring of dimension two and $M$ is a maximal Cohen-Macaulay $R$-module of rank two. Furthermore, since there exists a surjection $M\to \fkm; (f,g)\mapsto fg$, we get $h(M)=1<\infty$.
\end{rem}

\begin{ex}\label{ex28}
Let $k[[X, Y]]$ be the formal power series ring over a field $k$. Set $R=k[[X, XY, XY^2, XY^3]]$, $\fkm = (X, XY, XY^2, XY^3)$, and $I=(XY, XY^2)$. Then, $\tr_R(I)=\fkm$, that is, $\ell_R(R/\tr_R(I))=1$, but $h(I)=\infty$.
\end{ex}

\begin{proof}
Since $R$ is a local domain and $I\not\cong R$, $I$ does not have a free summand. Hence, $\tr_R(I)\subseteq \fkm$ by Fact \ref{r12}(3). On the other hand, we have $\tr_R(I)\supseteq I+Y^2 I = \fkm$ by considering the inclusion map $I\to R$ and the map $I \to R; a\mapsto Y^2 a$ for $a\in I$. It follows that $\tr_R(I)=\fkm$. 
On the other hand, we note that $I$ is a canonical module (see, for example, \cite[Theorem 6.3.5]{bh1}); hence, $I$ is maximal Cohen-Macaulay. Moreover, since $R$ is a domain, an ideal $I$ is of rank one.  Therefore, since $\dim R=2$, by Lemma \ref{addl29}, $h(I)=\infty$. 
\end{proof}

We consider the implication (3)$\Rightarrow$(2) of Proposition \ref{p21}. As seen in Example \ref{ex28}, this does not hold in general. However, it holds when $R$ is local and $\dim R=1$ (see Theorem \ref{1dim}). The following is a key to prove it.

\begin{prop}\label{p01}
Let $(R, \m)$ be a Noetherian local ring. Let $M$ be a finitely generated $R$-module. Let $\fkp_1, \fkp_2, \dots, \fkp_\ell\in \Spec (R)$. If $\tr_R(M)\not\subseteq \bigcup_{i=1}^\ell \fkp_i$, then there exists a homomorphism $f\in \Hom_R(M, R)$ such that $\Im f\not\subseteq \bigcup_{i=1}^\ell \fkp_i$.
\end{prop}

\begin{proof}
If $M$ has a free summand, then there is a surjection $M\to R$. Hence, we may assume that $M$ has no free summands. Then $\tr_R(M)\subseteq \m$ by Fact \ref{r12}(3). It follows that $\fkp_i\ne \fkm$ for all $1\le i \le \ell$. 
Assume that for each $f\in \Hom_R(M, R)$, $\Im f \subseteq \bigcup_{i=1}^\ell \p_i$. By the prime avoidance lemma, this turns out that $\Im f \subseteq \p_i$ for some $1\le i \le \ell$. 
It follows that 
\begin{align}\label{eq1}
\Hom_R(M, R)=\bigcup_{i=1}^{\ell} \Hom_R(M, \p_i).
\end{align}
We then obtain a contradiction by induction on $\ell$.
If $\ell=1$, then \eqref{eq1} shows that $\tr_R(M)\subseteq \p_1$. This contradicts our assumption $\tr_R(M)\not\subseteq \bigcup_{i=1}^\ell \fkp_i$. 

Let $\ell>1$ and assume that there exists a contradiction for $\ell-1$. We may assume that there exist the following elements
\begin{align*}
f_1\in \Hom_R(M, \p_1)\setminus \bigcup_{j\ne 1} \Hom_R(M, \p_j) \quad \text{and} \quad 
f_2\in \Hom_R(M, \p_2)\setminus \bigcup_{j\ne 2} \Hom_R(M, \p_j).
\end{align*}
Indeed, if there are no elements satisfying the above, we can remove $\Hom_R(M, \p_1)$ or $\Hom_R(M, \p_2)$ from the right hand union of \eqref{eq1}. 
Choose an element $a\in \m\setminus \bigcup_{i=1}^\ell \fkp_i$ and let $m\ge 0$. 
\begin{claim}\label{claim1}
$f_1 + a^m f_2 \not\in \Hom_R(M, \p_1)\cup \Hom_R(M, \p_2)$ for all $m\ge 0$.
\end{claim}

\begin{proof}[Proof of Claim \ref{claim1}]
If $f_1 + a^m f_2 \in \Hom_R(M, \p_2)$, then $f_1=(f_1 + a^m f_2)-a^m f_2\in \Hom_R(M, \p_2)$. This contradicts the choice of $f_1$.
Suppose that $f_1 + a^m f_2 \in \Hom_R(M, \p_1)$. Then $a^m f_2=(f_1 + a^m f_2)-f_1\in \Hom_R(M, \p_1)$, that is, $a^m \Im f_2\subseteq \p_1$. 
Since $\p_1$ is a prime ideal and $a^m\not\in \p_1$, $\Im f_2 \subseteq \p_1$. This proves $f_2\in \Hom_R(M, \p_1)$, a contradiction.
\end{proof}

By Claim \ref{claim1}, there is a contradiction if $\ell=2$. Suppose that $\ell>2$. We obtain that $f_1 + a^m f_2 \in \bigcup_{j=3}^\ell \Hom_R(M, \p_j)$ for all $m\ge 0$. Since $\ell$ is finite, there exist integers $3\le j \le \ell$ and $0 \le s < t$ such that $f_1 + a^s f_2$, $f_1 + a^t f_2\in \Hom_R(M, \p_j)$. It follows that 
\[
a^s(1-a^{t-s}) f_2=(f_1 + a^s f_2) - (f_1 + a^t f_2)\in \Hom_R(M, \p_j),
\]
that is, $a^s(1-a^{t-s}) \Im f_2\subseteq \p_j$. Since $1-a^{t-s}$ is a unit, $a^s \Im f_2\subseteq \p_j$. 
Hence, $\Im f_2\subseteq \p_j$ since $\p_j$ is a prime ideal and $a^s\not\in \p_j$. Thus, $f_2\in \Hom_R(M, \p_j)$. This contradicts the choice of $f_2$.

Hence, there exists at least one homomorphism $f\in \Hom_R(M, R)$ such that $\Im f\not\subseteq \bigcup_{i=1}^\ell \fkp_i$.
\end{proof}

\begin{thm}\label{1dim}
Let $(R, \m)$ be a Noetherian local ring of dimension one. Let $M$ be a finitely generated $R$-module. 
Then the following are equivalent.
\begin{enumerate}[\rm(1)] 
\item $h(M)<\infty$.
\item $\ell_R(R/\tr_R(M))<\infty$.
\item $S^{-1}R$ is a direct summand of $S^{-1}M$, where $\displaystyle S=R\setminus \bigcup_{\fkp\in \Spec (R)\setminus\{\fkm \}}\fkp$. 
\item $M_\fkp$ has a free summand for all $\fkp\in \Spec (R)\setminus\{\fkm \}$.
\end{enumerate}
\end{thm}

\begin{proof}
By Proposition \ref{p21}, we only need to prove that (2)$\Rightarrow$(1). 
We note that $\Spec (R)=\{\fkm\}\cup \Min (R)$ since $\dim R=1$. Suppose that $\ell_R(R/\tr_R(M))<\infty$. Since $\Spec (R)=\{\fkm\}\cup \Min (R)$, this is equivalent to saying that $\tr_R(M)\not\subseteq \bigcup_{\fkp\in \Min (R)} \fkp$. Since $\Min (R)$ is a finite set, by Proposition \ref{p01}, there exists a homomorphism $f\in \Hom_R(M, R)$ such that $\Im f\not\subseteq \bigcup_{\fkp\in \Min (R)} \fkp$. It follows that $h(M)\le \ell_R(R/\Im f)<\infty$. 
\end{proof}

The next purpose is to prove Theorem \ref{t26}, which asserts that Maitra's characterization of regular partial trace ideals (\cite[Theorem 2.5]{sm}) holds in more generality.

\begin{lem}\label{l110}
Let $(R, \fkm)$ be a Noetherian local ring. Let $I$ be a regular ideal such that $h(I)<\infty$. If $J$ is a partial trace ideal of $I$, then $J\cong I$. 
\end{lem}

\begin{proof}
Since $I$ is regular, $\Hom_R(I, R)\cong R:I$. Hence, there exists $\alpha\in R:I$ such that $J=\alpha I$ and $\ell_R(R/J) =h(I)<\infty$. Write $\alpha=a/b$ for $a,b\in R$ and $b$ is a non-zerodivisor of $R$. Assume that $a$ is a zerodivisor and choose a nonzero element $c\in R$ such that $ac=0$. Then $cJ=c\alpha I=(ac/b)I=0$. Hence, $J\subseteq \fkp$ for some $\fkp\in \Ass (R)$. On the other hand, since $I$ is regular, we may assume that $\depth R>0$ (otherwise, $I=R$). It follows that $\ell_R(R/J)\ge \ell_R(R/\fkp)=\infty$, which is a contradiction. Hence, $a$ is a non-zerodivisor of $R$ and thus $J=\alpha I \cong I$. 
\end{proof}

\begin{lem}\label{f26}
Let $(R, \fkm)$ be a Cohen-Macaulay local ring of dimension one, and let $I$ be a regular ideal of $R$, and $(a)$ is a reduction of $I$. Then, $\ol{I} = \ol{(a)} = a\ol{R}\cap R = I\ol{R}\cap R$.
\end{lem}

\begin{proof}
By \cite[Propositions 1.5.2 and 1.6.1]{HS}, we have 
\[
a\ol{R}\cap R = \ol{a\ol{R}}\cap R = \ol{(a)}.
\] 
Since $(a)$ is a reduction of $I$, $(a)\subseteq I \subseteq \ol{(a)} = a\ol{R}\cap R \subseteq a\ol{R}$. It follows that $R \subseteq a^{-1} I \subseteq \ol{R}$. Multiplying $\ol{R}$ in the inclusions, we get $a^{-1} I\ol{R} = \ol{R}$, that is, $I\ol{R} = a\ol{R}$. Hence, we get $a\ol{R}\cap R = I\ol{R}\cap R$. The equality $\ol{I} = \ol{(a)}$ follows from the assumption that $(a)$ is a reduction of $I$ (see \cite[Remark 2.3]{HS}). 
\end{proof}

\begin{thm}\label{t26}
Let $(R, \fkm)$ be a Cohen-Macaulay local ring of dimension one, and let $I$ be a regular ideal of $R$. Consider the following conditions.
\begin{enumerate}[\rm(1)] 
\item $I$ is a partial trace ideal. 
\item $I$ is a partial trace ideal of $I$.
\item $R:I\subseteq \ol{R}$.
\item $I$ is a reduction of $\tr_R(I)$, that is, $\tr_R(I)^{n+1}=I\tr_R(I)^n$ for some $n>0$.
\end{enumerate}
Then, {\rm (1)$\Leftrightarrow$(2)$\Leftarrow$(3)$\Rightarrow$(4)} hold. 
Furthermore, if all $\fkm$-primary ideals have a principal reduction (for example, the residue field $R/\fkm$ is infinite or $\ol{R}$ is local with the same residue field as that of $R$, see Remark \ref{tuesday}), then all the conditions are equivalent. 
\end{thm}

\begin{proof}
(1)$\Leftrightarrow$(2): This is clear and written in \cite[Remark 2.3]{sm}. 

(3)$\Rightarrow$(4):
Since $R\subseteq R:I\subseteq \ol{R}$, we have $I \subseteq \tr_R(I)=(R:I)I\subseteq I\ol{R}\cap R=\ol{I}$ by Lemma \ref{f26}. It follows that $I$ is a reduction of $\tr_R(I)$ by \cite[Corollary 1.2.5]{HS}.

(3)$\Rightarrow$(2): Let $S:=R[R:I]$ be the extension ring of $R$ with all elements of $R:I$. We note that $S$ is finitely generated as an $R$-module since $R:I\subseteq \ol{R}$. Assume the contrary, that is, there exists a partial trace ideal $J$ of $I$ such that $\ell_R(R/J)<\ell_R(R/I)$. By Lemma \ref{l110}, $J\cong I$. Let $\alpha\in \rmQ(R)$ such that $J=\alpha I$. Write $\alpha=a/b$ where $a,b\in R$. We have $aI=bJ$. Consider the following diagram of inclusions
\begin{align}\label{eq261}
\begin{split}
\xymatrix{
& R   &  \\
I \ar@{-}[ur] &  & J \ar@{-}[ul] \\
& aI=bJ. \ar@{-}[ur] \ar@{-}[ul] 
}
\end{split}
\end{align}
Here, since $I$ and $J$ are Cohen-Macaulay $R$-modules of rank one, $\ell_R(I/aI) = \ell_R(R/(a))$ and $\ell_R(J/bJ) = \ell_R(R/(b))$ (\cite[Theorem 14.8]{matsu}). 
Therefore, by the above diagram and the inequality $\ell_R(R/J)<\ell_R(R/I)$, we obtain that 
\[
\ell_R(R/(a))=\ell_R(I/aI) < \ell_R(J/bJ) =\ell_R(R/(b)).
\] 
Since $S$ is also a Cohen-Macaulay $R$-module of rank one, we obtain that $\ell_R(S/aS)=\ell_R(R/(a))$ and $\ell_R(S/bS)=\ell_R(R/(b))$; hence, $\ell_R(S/aS) < \ell_R(S/bS)$. 

On the other hand, since $\alpha I=J\subseteq R$, we have $\alpha \in R:I\subseteq S$. It follows that $\alpha S\subseteq S$, that is, $aS\subseteq bS$. In particular, $\ell_R(S/aS) \ge  \ell_R(S/bS)$.
This is a contradiction. Hence, there is no partial trace ideal $J$ of $I$ such that $\ell_R(R/J)<\ell_R(R/I)$.

(2)$\Rightarrow$(3): Suppose that all $\fkm$-primary ideals have a principal reduction. 
We choose a non-zerodivisor $b\in I$ of $R$. By noting that $b(R:I)$ is a regular ideal of $R$, there exists a minimal reduction $(a)$ of $b(R:I)$. Then $(a)\subseteq b(R:I) \subseteq \ol{(a)} = a\ol{R}\cap R \subseteq a\ol{R}$ by Lemma \ref{f26}. It follows that 
\begin{align}\label{eq213}
R\subseteq (b/a)(R:I)=R:(a/b)I \subseteq \ol{R}.
\end{align}
Set $J:=(a/b)I$. We note that $J$ is an ideal of $R$ since $a/b\in R:I$. We further obtain that 
\[
\ell_R(R/J) = h(J) = h(I) = \ell_R(R/I),
\]
where the first equality follows by applying the result (3)$\Rightarrow$(2) to $J$ (see \eqref{eq213}), the second equality follows from $J\cong I$, and the third equality follows from the assumption. Note that $bJ=aI$ by the definition of $J$. Then, by considering the same diagram as \eqref{eq261} and noting the equations $\ell_R(R/(a))=\ell_R(I/aI)$ and $\ell_R(J/bJ) =\ell_R(R/(b))$, we have $\ell_R(R/(a))=\ell_R(R/(b))$. On the other hand, since $b/a \in (b/a)(R:I)\subseteq \ol{R}$ by \eqref{eq213}, the extension ring $R[b/a]$ of $R$ is finitely generated as an $R$-module. Set $S:=R[b/a]$. Since $S$ is a Cohen-Macaulay $R$-module of rank one, we obtain that $\ell_R(S/aS) = \ell_R(R/(a))=\ell_R(R/(b)) = \ell_R(S/bS)$. On the other hand, since $b/a\in R[b/a] = S$, we have $(b/a)S\subseteq S$, that is, $bS\subseteq aS$. Combining the inclusion $bS\subseteq aS$ with the equality $\ell_R(S/aS) = \ell_R(S/bS)$, we have $aS = bS$, that is, $(b/a)S=S$. By multiplying $S$ to the inclusion $(b/a)(R:I) \subseteq \ol{R}$, we obtain that
\[
(b/a)(R:I)S \subseteq S\ol{R} = \ol{R}.
\] 
Since $R:I\subseteq (R:I)S = (b/a)(R:I)S$, we obtain that $R:I\subseteq \ol{R}$ as desired.

(4)$\Rightarrow$(3): Suppose that all $\fkm$-primary ideals have a principal reduction. We choose a minimal reduction $(a)$ of $I$. Then, $(a)$ is also a reduction of $\tr_R(I)$ by the assumption. It follows that $\tr_R(I) =(R:I)I\subseteq \ol{(a)}\subseteq a\ol{R}$ by Lemma \ref{f26}. Thus, $(R:I)a^{-1}I\subseteq \ol{R}$. We note that $a^{-1}I \ol{R} =\ol{R}$. Indeed, we have $R\subseteq a^{-1}I\subseteq \ol{R}$ since $(a)\subseteq I\subseteq \ol{(a)}\subseteq a\ol{R}$ by Lemma \ref{f26}. By multiplying $\ol{R}$ to the inclusion, we have $a^{-1}I \ol{R} =\ol{R}$. Therefore, by multiplying $\ol{R}$ to $(R:I)a^{-1}I\subseteq \ol{R}$,  we obtain that 
\[
(R:I)\ol{R} = (R:I)a^{-1}I\ol{R}\subseteq  \ol{R}.
\]
It follows that $R:I\subseteq \ol{R}:\ol{R} = \ol{R}$.
\end{proof}

\begin{rem}
Maitra proved the equivalence of (2)$\Leftrightarrow$(3) of Theorem \ref{t26} under the additional assumptions that $R$ is a domain, $\widehat{R}$ is reduced, and $\ol{R}$ is a discrete valuation ring (\cite[Theorem 2.5]{sm}). Especially, the assumption that $\ol{R}$ is a discrete valuation ring is crucial in his proof. Indeed, Maitra repeatedly asked in his paper if the characterization holds without his additional assumptions, see \cite[after Theorem B, after Theorem 2.5, and Conclusion (2)]{sm}. 
\end{rem}

\begin{cor}\label{c113}
Let $(R, \fkm)$ be a Cohen-Macaulay local ring of dimension one and let $I$ be a regular ideal of $R$. Let $b\in I$ be a non-zerodivisor of $R$. 
Suppose that we can choose a reduction $(a)$ of $b(R:I)$ (for example, the residue field $R/\fkm$ is infinite or $\ol{R}$ is local with the same residue field as that of $R$, see Remark \ref{tuesday}). Then, $J:=(a/b)I$ is a partial trace ideal of $I$.
\end{cor}

\begin{proof}
Since $(a)\subseteq b(R:I) \subseteq \ol{(a)} = a\ol{R}\cap R$, we have 
\[
R\subseteq (b/a)(R:I)=R:(a/b)I \subseteq \ol{R}.
\]
Hence, $R:J\subseteq \ol{R}$ and the assertion follows from Theorem \ref{t26} (3)$\Rightarrow$(2).
\end{proof}

The following provides a formula to compute $h(I)$ for analytically irreducible rings. Let $(R, \fkm)$ be a Cohen-Macaulay local ring of dimension one. Suppose that $\ol{R}$ is a discrete valuation ring with the maximal ideal $\fkn$. Suppose that $R/\fkm\cong \ol{R}/\fkn$. For a fractional ideal $J$, we denote by 
\[
v(J):=\min\{ v(x) : x\in J\},
\]
the {\it value} of $J$, where $v(x)$ denotes the {\it normalized valuation} associated to $\ol{R}$ and is defined as follows:
\[
v:\rmQ(R) \to \mathbb{Z}\cup\{\infty\}; a/b \mapsto \ell_{\ol{R}}(\ol{R}/a\ol{R}) - \ell_{\ol{R}}(\ol{R}/b\ol{R})
\] 
for $a,b\in R$ such that $b$ is a non-zerodivisor of $R$.

\begin{rem}\label{tuesday}
Let $(R, \fkm)$ be a Cohen-Macaulay local ring of dimension one and let $I$ be a regular ideal of $R$. Suppose that $\ol{R}$ is a discrete valuation ring with the maximal ideal $\fkn$ and $R/\fkm\cong \ol{R}/\fkn$. If $b\in I$ with $v(b)=v(I)$, then $(b)$ is a reduction of $I$. 
\end{rem}

\begin{proof}
By the choice of $b$, we have $I \overline{R}=b\overline{R} =\fkn^{v(b)}$. It follows that 
\[
(b)\subseteq I\subseteq I \overline{R} \cap R = b\ol{R} \cap R = \ol{(b)} 
\]
by Lemma \ref{f26}. Hence, $(b)$ is a reduction of $I$ (\cite[Corollary 1.2.5]{HS}). 
\end{proof}

\begin{lem}\label{l219e}
Let $(R, \fkm)$ be a Cohen-Macaulay local ring of dimension one and let $I$ be a regular ideal of $R$. Suppose that $\ol{R}$ is a discrete valuation ring with the maximal ideal $\fkn$ and $R/\fkm\cong \ol{R}/\fkn$. Let $a,b$ be non-zerodivisors of $R$ such that $a/b\in R:I$. Set $\alpha =a/b$. Then, $\ell_R(R/\alpha I)=\ell_R(R/I) +v(\alpha)$. 
\end{lem}

\begin{proof}
Set $J=\alpha I$. Then, $aI=bJ$ by definition. By considering the diagram \eqref{eq261} and the equations $\ell_R(I/aI) = \ell_R(R/(a))$ and $\ell_R(J/bJ) = \ell_R(R/(b))$, we get $\ell_R(R/\alpha I) = \ell_R(R/I) + \ell_R(R/(a)) - \ell_R(R/(b))$. 

Since $\ol{R}$ is a finitely generated $R$-module of rank one, we also have 
\begin{center}
$\ell_R(R/(a)) = \ell_R(\ol{R}/a\ol{R}) =v(a)$ \quad and \quad $\ell_R(R/(b)) = \ell_R(\ol{R}/b\ol{R})= v(b)$. 
\end{center}

Hence, the assertion holds.
\end{proof}

\begin{cor}
Let $(R, \fkm)$ be a Cohen-Macaulay local ring of dimension one and let $I$ be a regular ideal of $R$. Suppose that $\ol{R}$ is a discrete valuation ring with the maximal ideal $\fkn$ and $R/\fkm\cong \ol{R}/\fkn$. Then, $h(I)=\ell_R(R/I) + v(R:I)$.
\end{cor}

\begin{proof}
Choose a non-zerodivisor $b\in I$ of $R$.  We also choose $a\in b(R:I)$ such that $v(a)=v(b(R:I))$. By Remark \ref{tuesday}, $(a)$ is a reduction of $b(R:I)$. Set $J:=(a/b)I$. By Corollary \ref{c113}, $J$ is a partial trace ideal of $I$. By Lemma \ref{l219e},
\begin{align}
h(I)=\ell_R(R/J) = \ell_R(R/I) + v(a/b) = \ell_R(R/I) + v(a)-v(b).
\end{align}
On the other hand, we have $v(a)= v(b(R:I)) = v(b) + v(R:I)$. Hence, we obtain $h(I)=\ell_R(R/I) + v(R:I)$ as desired.
\end{proof}

The following answers a question of Maitra, which is posed in his another paper \cite[Question 5.1]{Mai2}. That is, for a given regular ideal $I$, how many partial trace ideals can exist?

\begin{cor}\label{corcor}
Let $(R, \fkm)$ be a Cohen-Macaulay local ring of dimension one and let $I$ be a regular ideal of $R$. Suppose that $\ol{R}$ is a discrete valuation ring with the maximal ideal $\fkn$ and $R/\fkm\cong \ol{R}/\fkn$. Suppose that $J$ is a partial trace ideal of $I$. 
Then, 
\begin{align*}
\{\text{partial trace ideals of $I$}\} = \{\alpha J \mid \alpha\in R:J, v(\alpha) =0\}.
\end{align*}
\end{cor}

\begin{proof}
($\supseteq$): This follows from the equations $\ell_R(R/\alpha J) =\ell_R(R/J) +v(\alpha) =\ell_R(R/J)=h(I)$ by Lemma \ref{l219e}.

($\subseteq$): Let $L$ be a partial trace ideal of $I$. By Lemma \ref{l110}, $L\cong I\cong J$. It follows that there exists $\alpha\in \rmQ(R)$ such that $L=\alpha J$. Since $\alpha J =L\subseteq R$, $\alpha \in R:J$. Furthermore, by Lemma \ref{l219e}, we get 
\[
\ell_R(R/J)=h(I) =\ell_R(R/L) =\ell_R(R/J) +v(\alpha).
\]
Hence, $v(\alpha)=0$ as desired.
\end{proof}


\section{Discussions on the {\it h}-invariant of the canonical module}\label{sec3}

Throughout this section, unless otherwise noted, let $(R,\m)$ be a Cohen-Macaulay local ring admitting the canonical module $\omega_R$. In this section we focus our attention on the partial trace ideal of the canonical module and $h(\omega_R)$. 

\begin{rem} \label{pp31} (cf. \cite[Theorem 4.6(1)]{Mai2})
Consider the following conditions:
\begin{enumerate}[\rm(1)]  
\item $R$ is Gorenstein. 
\item $h(\omega_R)=0$.
\item $h(\omega_R)=\ell_R(R/\tr_R(\omega_R))<\infty$.
\end{enumerate}
Then, $(1)\Leftrightarrow (2) \Rightarrow (3)$ hold. If $\dim R>0$, then all the conditions are equivalent. 
\end{rem}     

\begin{proof} 
$(1)\Leftrightarrow (2)$ and $(1)\Rightarrow (3)$ are clear. 


$(3)\Rightarrow (1)$: Suppose that $\dim R>0$. Note that $\ell_R(R/\tr_R(\omega_R))<\infty$ implies $\tr_R(\omega_R)$ is $\m$-primary, hence $R$ is Gorenstein on the punctured spectrum, hence $R$ is generically Gorenstein since $\dim R>0$. The assumption also means that $f(\omega_R)=\tr_R(\omega_R)$ for some $f \in \hom_R(\omega_R,R)$, i.e., $R$ is of Teter type in the sense of \cite{teter}. By \cite[Theorem 1.1]{teter}, $R$ is Gorenstein. 
\end{proof}

The implication (3)$\Rightarrow$(2) of Remark \ref{pp31} does not hold if $\dim R=0$ as shown in \cite{teter}.

\begin{ex} (\cite[Example 2.10]{teter})
Let $K[x,y]$ be the polynomial ring over a field $K$. Set $R=K[x,y]/(x^4,y^4,x^2y^2)$. Then $\tr_R(\omega_R)=(x^2, y^2)$ and there exists a surjective homomorphism $\omega_R \to \tr_R(\omega_R)$. Hence, $h(\omega_R) = \ell_R(R/\tr_R(\omega_R)) = 4$, but, $R$ is not Gorenstein. 
\end{ex}

If $R$ is not Gorenstein, the finiteness of $h(\omega_R)$ requires that $\dim R\le 1$ as follows.

\begin{rem}\label{p33}
Suppose that $\dim R>0$ and $R$ is not Gorenstein. 
If $h(\omega_R)<\infty$, then $\dim R=1$ and $R$ is generically Gorenstein, i.e., $R_\fkp$ is Gorenstein for all $\fkp\in \Ass (R)$. 
\end{rem}

\begin{proof}
By Proposition \ref{p21}(1)$\Rightarrow$(4), $R_\fkp$ is Gorenstein for all $\fkp\in \Ass(R)$. That is, $R$ is generically Gorenstein and $\omega_R$ is of rank one (\cite[Proposition 3.3.18]{bh1}). By Lemma \ref{addl29}, $\dim R\le 1$. 
\end{proof}

Due to Remarks \ref{pp31} and \ref{p33}, in dimension $\le 1$, $h(\omega_R)$ plays the role of an invariant measuring how close to being Gorenstein. Thus, one may expect a characterization of rings having a small number of $h(\omega_R)$. With this perspective, the following results are essentially known although the $h$-invariant is a new notion introduced by Maitra and thus is not explicitly stated in terms of the $h$-invariant in the literature.

\begin{fact}\label{hvk}
\begin{enumerate}[\rm(1)] 
\item (\cite[Theorem 2.5]{HV}): Assume that $2$ is a unit and $\mathrm{Soc}(R)\subseteq \fkm^{2}$. Then the following are equivalent.
\begin{enumerate}[\rm(a)]
    \item $h(\omega_R)=1$.
    \item There exists an Artin Gorenstein ring $S$ such that $R \cong S/(0):\fkm_{S}$. 
\end{enumerate}
\item (\cite[Theorem 1.4]{Ko}): Suppose that $\dim R=1$. Then the following are equivalent.
\begin{enumerate}[\rm(a)] 
\item $h(\omega_R)\le 2$.
\item Either $R$ is Gorenstein or there exists a Gorenstein subring $S$ of $R$ such that $\ell_{S}(R/S)=1$ and $\mathrm{Q}(S)=\mathrm{Q}(R)$.
\item Either $R$ is Gorenstein or there exists a one-dimensional Gorenstein local ring $(S,\fkn)$ such that $R \cong \mathrm{End}_{S}(\fkn)$.
\end{enumerate}
\end{enumerate}
\end{fact}

\begin{rem}
If $\dim R=1$, then $h(\omega_R)\ne 1$. Indeed, if $h(\omega_R) = 1$, then there exists a surjection $\omega_R \to \fkm$. This is impossible if $\dim R=1$, see \cite[Theorem 1.1]{teter}. 
\end{rem}

Next, we want to state another characterization of rings with $h(\omega_R)=2$. First we need some preparatory results and notions.

\begin{dfn} (\cite[Definition 2.1]{hw}) 
Let $(R, \fkm)$ be a Noetherian local ring. Then, an ideal $I$ is called {\it weakly $\fkm$-full} if $(\m I:_R \m)=I$. 
\end{dfn}

By Nakayama's lemma, $(I\m:_R \m)\subseteq \m$ if neither $I=R$ nor $R$ is a field. Since $\m \subseteq (\m^2:_R \m)$, $\m$ is weakly $\m$-full if $R$ is not  a field.

For latter use, we also recall that an ideal $I$ of $R$ is called {\it Burch}  if $\m I\neq \m(I:_R \m)$ (\cite[Definition 2.1]{DKT}). 

\begin{prop}\label{weakm} 
Let $(R, \fkm)$ be a Noetherian local ring, but not a field. Let $I$ be an ideal of $R$ such that $R/I$ is a Gorenstein ring. 
If $I$ is not weakly $\m$-full and $\m^2\subseteq I$, then $\m^2=I\m$ and $\mu_R(\m)=1+\mu_R(I)$.  
\end{prop}

\begin{proof} 
Since $I$ is not weakly $\m$-full, $I\ne \m$. Since $\m^2 \subseteq I$, $R/I$ is an Artinian ring of minimal multiplicity. Since $R/I$ is also Gorenstein and non-regular (as $R/I$ is not a field), so we get that $R/I$ is an Artinian hypersurface; hence, the maximal ideal of $R/I$ is principal. It follows that $\m/I=[(f)+I]/I$ for every $f\in R$ such that $f \in \m \setminus (\fkm^2+I)$. By the assumption that $\m^2\subseteq I$, $\m^2+I=I$. Since $I$ is not weakly $\m$-full, so we can choose $g\in (\m I:_R \m)\setminus I$. 
Then, $\fkm=(g) +I$. Therefore, $\fkm^2=\fkm g +\fkm I = \fkm I$ by the choice of $g$. 
It follows that the canonical homomorphism $I/\fkm I \to \fkm /\fkm^2$ is injective since $\fkm^2\cap I=\fkm^2 =\fkm I$. Hence, $\mu_R(\m)=1+\mu_R(I)$.
\end{proof}

We recall that $R$ is called {\it nearly Gorenstein} if $\m \subseteq \tr(\omega_R)$ (\cite[Definition 2.2]{hbs}).

\begin{prop}\label{nn34} 
Suppose $\dim R>0$. Then the following hold. 
\begin{enumerate}[\rm(1)] 
\item $h(\omega_R)\ne 1$. 
\item If $h(\omega_R)=2$, then $\m=\tr_R(\omega_R)$, that is, $R$ is nearly Gorenstein, but not Gorenstein.
\end{enumerate}
\end{prop}

\begin{proof}
Suppose that $\dim R>0$. Note that $0\le \ell_R(R/\tr_R(\omega_R)) \le  h(\omega_R)$. 

(1): Suppose that $h(\omega_R)=1$. $R$ is not Gorenstein by Proposition \ref{pp31}. Since $R=\tr_R(\omega_R)$ if and only if $R$ is Gorenstein, we get $\ell_R(R/\tr_R(\omega_R)) =  h(\omega_R) =1$ by the above inequality. This is a contradiction by Proposition \ref{pp31} (1) $\Leftrightarrow$ (3) since $\dim R>0$. 

(2): $R$ is not Gorenstein by Proposition \ref{pp31}. By the same reason as above, we have $0< \ell_R(R/\tr_R(\omega_R)) <  h(\omega_R)\le 2$. It follows that $\m=\tr_R(\omega_R)$.
\end{proof}

Finally, we also need the notion of almost Gorenstein rings and its connection with nearly Gorenstein rings.

\begin{dfn}{\rm (\cite[Theorem 3.11]{GMP})}\label{d37}
Let $(R, \fkm)$ be a one-dimensional Cohen-Macaulay local ring having the canonical module $\omega_R$. Suppose that there exists an ideal such that $I\cong \omega_R$ and its minimal reduction $(b)\subseteq I$ (e.g., $R$ is generically Gorenstein and the residue field $R/\fkm$ is infinite). Then $R$ is called {\it almost Gorenstein} if $\fkm I=\fkm b$.
\end{dfn}

\begin{fact}{\rm (\cite[Theorem 6.6]{hbs})}\label{f38}
Let $(R, \fkm)$ be a one-dimensional Cohen-Macaulay local ring having the canonical module $\omega_R$. Suppose that the residue field is infinite. Consider the following condition. 
\begin{enumerate}[\rm(1)] 
\item $R$ is nearly Gorenstein.
\item $R$ is almost Gorenstein. 
\end{enumerate}
Then, (2)$\Rightarrow$(1) holds. (1)$\Rightarrow$(2) holds if $R$ has minimal multiplicity. 
\end{fact}

\begin{fact} (cf. \cite[Lemma 5.2]{dk}) \label{ega}
Let $(R, \fkm)$ be a one-dimensional Cohen-Macaulay local ring. Then, $R$ has minimal multiplicity if and only if $\fkm^2\cong \fkm$. 
\end{fact}

Now we are ready to state and prove a characterization of rings with $h(\omega_R)=2$.

\begin{thm}\label{h=2} Let $(R,\m)$ be a Cohen-Macaulay local ring of dimension $1$ admitting the canonical module $\omega_R$. Consider the following conditions:

\begin{enumerate}[\rm(1)]
    \item $h(\omega_R)=2$. 
    \item $R$ is not a hypersurface, there exists an ideal $I$ of $R$ such that $I\cong \omega_R$ and $\m^2\subseteq I\subseteq \m$. 
    \item $R$ is not a hypersurface, there exists an ideal $I$ of $R$ such that $I\cong \omega_R$ and $\m^2=I\m$.
    \item $\tr_R(\omega_R)=\m$ and $\mu_R(\m)=1+\mathrm{r}(R)$. 
\end{enumerate}
Then, $(1)\Leftrightarrow(2)\Leftrightarrow (3)\Rightarrow (4)$ hold true. Furthermore, if $R$ has minimal multiplicity and the residue field $R/\fkm$ is infinite, then all the four conditions are equivalent. 
\end{thm} 

\begin{proof} 
In any implication, we may assume that $R$ is not regular. 

$(1)\Rightarrow (2)$: By Proposition \ref{p33}, $R$ is generically Gorenstein. We can choose an ideal $J$ such that $J\cong \omega_R$ by \cite[Proposition 3.3.18]{bh1}. Let $I$ be a partial trace ideal of $J$. By Lemma \ref{l110}, $I\cong J \cong \omega_R$. It follows that $R/I$ is Artin Gorenstein Proposition \ref{p33}. On the other hand, since $\ell_R(R/I) = h(J) =h(\omega_R) =2$, $\fkm^2\subseteq I$. Hence, $\m^2\subseteq I\subseteq \m$ since $\ell_R(R/I) =2$. The assertion that $R$ is not a hypersurface is clear since $R$ is not Gorenstein. 

$(2)\Rightarrow (3)$: Notice that $R/I$ is a Gorenstein ring by \cite[Proposition 3.3.18]{bh1}. We will be done by Proposition \ref{weakm} if we can show that $I$ is not weakly $\m$-full. Suppose the contrary, that is, let $I$ be weakly $\m$-full. 

Since $I$ is $\m$-primary, so $I$ is Burch in $R$ by \cite[Lemma 4.3]{dk} . Then, by \cite[Lemma 3.7]{dk} (and the discussion following it), there is an $R$-regular (hence also $I$-regular) element $a\in \m$ such that $R/\m$ is a direct summand of $I/aI$. But $I/aI\cong \omega_R/a\omega_R$ has finite injective dimension over $R/aR$, hence $R/\m$ would have finite injective dimension over $R$, contradicting $R$ is non-regular. 

$(3)\Rightarrow (1)$: Note that $R/I$ has minimal multiplicity since $\m^2\subseteq I$. On the other hand, since $R/I$ is a Gorenstein ring (\cite[Proposition 3.3.18]{bh1}), $R/I$ is an Artinian hypersurface. Thus, $\m/I$ is a principal ideal of $R/I$. Then, since $\m^2\subseteq I$, we obtain that $\ell_R(R/I)=\ell_{R}((R/I)/(\fkm/I))+\ell_{R}(\fkm/I)=1+\mu_{R}(\m/I)=2$. 

It follows that $h(\omega_R)\le \ell(R/I)=2$. By Proposition \ref{nn34}(1), we get either $h(\omega_R)=0$ or $h(\omega_R)=2$. If $h(\omega_R)=0$, then $R$ is Gorenstein. Then $I\cong \omega_R\cong R$ and the hypothesis of (3) implies $\m^2=I\m\cong \m$, hence $R$ itself has minimal multiplicity by Fact \ref{ega}. Therefore, $R$ becomes a hypersurface, contrary to the hypothesis of (3). Thus $h(\omega_R)=2$.   

Hence, we prove $(1)\Leftrightarrow (2)\Leftrightarrow (3)$. Suppose that (1)-(3). $\tr_R(\omega_R)=\m$ follows by Proposition \ref{nn34}(2). Noting that $I$ is not weakly $\fkm$-full by the proof of $(2)\Rightarrow (3)$, $\mu_R(\m)=1+\mathrm{r}(R)$ follows by Proposition \ref{weakm} and the fact that $\textrm{r}(R)=\mu(\omega_R)=\mu(I)$.

$(4)\Rightarrow (1)$: Suppose that $R$ has minimal multiplicity and the residue field $R/\fkm$ is infinite. Then $R$ is almost Gorenstein by Fact \ref{f38}. Choose an $\fkm$-primary ideal $J$ such that $J\cong \omega_R$. We can choose a minimal reduction $(b)$ of $J$ since $R/\fkm$ is infinite. Then $\fkm J=\fkm b$ by Definition \ref{d37}. Let $(a)$ be a minimal reduction of $\fkm$ and set $I:=ab^{-1} J$. Then, $I$ is an ideal of $R$ since $I\subseteq \fkm b^{-1} J = b^{-1} \fkm b =\fkm$. It follows that 
\begin{align}\label{tired}
(a)\subseteq I \subseteq \fkm =\ol{(a)}
\end{align} 
since $(a)$ is a reduction of $\fkm$ and Fact \ref{f26}. On the other hand, we have $\ell_R(\fkm/(a))=\ell_R(R/(a)) -1 =\rme(R)-1 =\rmr(R)$, where the last equality follows from, for example, \cite[Fact 2.6]{HKS}. We further have 
\[
\ell_R(I/(a)) =\mu_R(I/(a)) =\mu_R(I) -1 =\mu_R(\omega_R)-1 =r(R)-1, 
\]
where the first equality follows from $\fkm^2=a\fkm\subseteq (a)$, the second equality follows from the fact that $(a)$ is a minimal reduction of $I$ (see \eqref{tired}). Therefore, we obtain that $\ell_R(\fkm/I) = \ell_R(\fkm/(a)) - \ell_R(I/(a)) = 1$. Therefore, 
\[
h(\omega_R) =h(I) \le \ell_R(R/I) =\ell_R(\fkm/I) +1 =2.
\]
Since $R$ is not Gorenstein, the assertion (1) holds by Proposition \ref{nn34}. 
\end{proof}

By Fact \ref{hvk}, it is natural to explore an upper bound of $h(\omega_R)$. For this purpose, we recall the notion of birational Gorenstein colength.

\begin{defn}
(\cite[Definition 1.3]{Ko}): Let $R$ be a one-dimensional Cohen-Macaulay local ring. Then, the invariant 
\[
bg(R):=\inf \{\ell_S(R/S) \mid S\subseteq R \text{ is a module finite}, \rmQ(S)=\rmQ(R), \text{ and $S$ is Gorenstein}\}
\]
is called the {\it birational Gorenstein colength} of $R$. 
\end{defn}

Using this notation, the result of Fact \ref{hvk} (2)(a) $\Leftrightarrow$ (b) can be rephrased as $h(\omega_R)\le 2$ $\Leftrightarrow$ $bg(R)\le 1$. With this observation, the following result provides a generalization of one direction of the result.

\begin{thm}\label{tbg}
Let $(R,\m)$ be a one-dimensional Cohen-Macaulay local ring having a canonical module $\omega_R$. If $R$ is generically Gorenstein, then $h(\omega_R)\leq 2bg(R)$. 
\end{thm}

\begin{proof}
We may assume $bg(R)<\infty$. Hence, there exists a Gorenstein subring $S$ of $R$ such that $Q(S)=Q(R)$ and $\ell_S(R/S)=bg(R)<\infty$. Hence, $\dim S=\dim R=1$. Note that $S$ is local since $\m\cap S$ is the unique maximal ideal of $S$. Hence, $\omega_R \cong \Hom_S(R,S)\cong (S: R)$, see \cite[Theorem 3.3.7(b)]{bh1}. Since $(S:R)$ is an ideal of $R$, we obtain that 
\begin{align}\label{eq361}
h(\omega_R)\leq \ell_R(R/(S:R))=\dfrac{\ell_S(R/(S:R))}{\ell_S(R/\m)}, 
\end{align}
where the last equality holds since we have a module finite extension $S/(\m\cap S)\subseteq R/\m$. On the other hand, by applying the $S$-canonical dual $\Hom_S(-, S)$ to the exact sequence $0 \to S\to R \to R/S \to 0$, we have $S/(S:R) \cong \Ext_S^1(R/S, S)$. By \cite[Lemma 3.1.16 and Proposition 3.2.12]{bh1}, it follows that $\ell_S(S/(S:R))=\ell_S(R/S)$. Hence, we get $\ell_S(R/(S:R))=\ell_S(S/(S:R))+\ell_S(R/S)=2\ell_S(S/(S:R))=2bg(R)$. Combining this equation with the inequality \eqref{eq361}, we complete the proof. 
\end{proof}

  \section{three-generated numerical semigroup rings}\label{sec4}
  
In this section we give a formula for $h(\omega_R)$ in the case where $R$ is a numerical semigroup ring generated by $3$ elements. First of all, let us fix the notation, according to the terminology of numerical semigroups.

\begin{setting}
Let $\ell >0$ and let $0 < a_1, a_2, \ldots, a_\ell \in \mathbb Z$ be positive integers such that $\mathrm{GCD}~(a_1, a_2, \ldots, a_\ell)=1$. We set 
\[
H = \left<a_1, a_2, \ldots, a_\ell\right>=\left\{\sum_{i=1}^\ell c_ia_i \ \middle| \  0 \le c_i \in \mathbb{Z} ~\text{for~all}~1 \le i \le \ell \right\}
\]
and call it the {\it numerical semigroup} generated by $a_1, \dots, a_\ell$. Let $k[t]$ be the polynomial ring over a field $k$. We set
$$R = k[H] = k[t^{a_1}, t^{a_2}, \ldots, t^{a_\ell}]$$
in $k[t]$ and call it the {\it numerical semigroup ring} of $H$ over $k$. The ring  $R$ is a one-dimensional Cohen-Macaulay graded domain such that $\overline{R} = k[t]$ and $\m = (t^{a_1},t^{a_2}, \ldots, t^{a_\ell} )$. We call
\begin{align*}
	\mathrm{PF}(H)=&\{ \alpha \in \mathbb{N}\setminus H \mid \text{$\alpha+h\in H$ for all  $h\in H\setminus\{0\}$}\}\\
	=&\{ \alpha \in \mathbb{N}\setminus H \mid \fkm t^\alpha\subseteq R\}
\end{align*}
the set of {\it pseudo-Frobenius numbers of $H$}.

Let $P=k[X_1, \dots, X_\ell]$ be the polynomial ring over $k$. We consider $P$ to be a $\mathbb{Z}$-graded ring such that $P_0=k$ and $\deg X_i=a_i$ for all $1\le i \le \ell$. Let $\varphi: P \to R$ denote the homomorphism of graded $k$-algebra defined by $\varphi(X_i)=t^{a_i}$ for each $1\le i \le \ell$. 

\end{setting}

In this section, we focus on the case of $\ell=3$. For simplicity, let us write $X, Y, Z$ instead of $X_1$, $X_2$, $X_3$, respectively. We then have the following.

\begin{fact} {\rm (\cite{Herzog})}\label{f42}
	Suppose that $R$ is not a Gorenstein ring. Then, 
	$$\Ker \varphi=I_2\begin{pmatrix}   
		X^{\alpha} & Y^{\beta} & Z^{\gamma} \\
		Y^{\beta'} & Z^{\gamma'} & X^{\alpha'}
	\end{pmatrix}$$
	for some $0< \alpha, \beta, \gamma, \alpha', \beta', \gamma'\in \mathbb Z$, where $I_2(\mathbb{M})$ is an ideal generated by the $2 \times 2$ minors of the matrix  $\mathbb{M}$. 
\end{fact}

In what follows, we suppose that $\ell=3$ and $R$ is not a Gorenstein ring. Let 
\begin{center}
	$\Delta_1 = Z^{\gamma+\gamma'}-X^{\alpha'}Y^{\beta}$, $\Delta_2=X^{\alpha+\alpha'}-Y^{\beta'}Z^{\gamma}$, \ and\  $\Delta_3=Y^{\beta+\beta'}-X^{\alpha}Z^{\gamma'}$. 
\end{center}
By Hilbert-Burch's theorem,  a graded minimal $P$-free resolution of $R$ is given by 
\begin{align}\label{seq1}
	0\longrightarrow \begin{matrix} P(-m )\\ \oplus\\ P(-n)\end{matrix} \xrightarrow{\left[ \begin{smallmatrix}
			X^{\alpha} & Y^{\beta'}\\
			Y^{\beta} & Z^{\gamma'}\\
			Z^{\gamma} & X^{\alpha'}
		\end{smallmatrix} \right]} \begin{matrix} P(-d_1)\\ \oplus\\ P(-d_2)\\ \oplus\\ P(-d_3)\end{matrix} \xrightarrow{\left[\Delta_1~-\Delta_2~\Delta_3\right]} P\overset{\varphi }{\longrightarrow } R \longrightarrow 0,
\end{align}
where $d_1 = \deg\Delta_1 = a_3(\gamma + \gamma')$, $d_2 = \deg\Delta_2 = a_1(\alpha + \alpha')$, $d_3 = \deg\Delta_3 = a_2(\beta + \beta')$, $m = a_1\alpha + d_1 = a_2\beta + d_2 = a_3\gamma + d_3$, and $n = a_1\alpha' + d_3 = a_2\beta' + d_1 = a_3\gamma' + d_2$. 

Let $\omega_P=P(-d)$ be the graded canonical module of $P$, where $d$ is the $a$-invariant of $P$. By applying the functor $\Hom_P(-, \omega_P)$ to the exact sequence \eqref{seq1}, we obtain that 
  \begin{equation}\label{resolution}
  	\begin{matrix} P(d_1-d)\\ \oplus\\ P(d_2-d)\\ \oplus\\ P(d_3-d)\end{matrix}\xrightarrow{\left[ \begin{smallmatrix} X^{\alpha} & Y^{\beta} & Z^{\gamma} \\
  			Y^{\beta'} & Z^{\gamma'} & X^{\alpha'}
  		\end{smallmatrix} \right]} \begin{matrix} P(m  -d)\\ \oplus\\ P(n-d)\end{matrix}\longrightarrow \omega_R\longrightarrow 0
  \end{equation}
  since $\omega_R\cong \Ext_P^2(R, \omega_P)$ (\cite[Proposition 3.6.12]{bh1}). On the other hand, it is also known that 
  \[
  \omega_R\cong \sum_{f\in \mathrm{PF}(H)} Rt^f
  \]
  by \cite[Example (2.1.9)]{GW}. Note that $\mathrm{PF}(H)=\{f_1, f_2\}$ for some positive integers $f_1< f_2$ by Fact \ref{f42}. 
  Therefore, by looking at the degrees in \eqref{resolution}, we have 
  \[
  f_2-f_1=|n-m|=\begin{cases}
|a_1 \alpha  - a_2\beta'|\\
|a_2 \beta  - a_3 \gamma'|\\
|a_3 \gamma  - a_1 \alpha'|
\end{cases}.
  \]

It follows that $\omega_R\cong (1, t^{f_2-f_1})$ by forgetting the shift. 
Hence, we have the following (non-graded) isomorphisms:  

\begin{prop}\label{canon}
$\omega_R\cong (t^{a_1\alpha}, t^{a_2\beta'}) \cong (t^{a_2 \beta}, t^{a_3 \gamma'}) \cong (t^{a_3 \gamma}, t^{a_1 \alpha'})$. 
\end{prop}

After rearranging $a_1, a_2, a_3$ and reordering the form of the matrix $\left(\begin{smallmatrix}   
	X^{\alpha} & Y^{\beta} & Z^{\gamma} \\
	Y^{\beta'} & Z^{\gamma'} & X^{\alpha'}
\end{smallmatrix}\right)$, we may assume that $a_1 \alpha$ is the smallest integer among $a_1\alpha$, $a_2 \beta$, $a_3 \gamma$, $a_2\beta'$, $a_3\gamma'$, and $a_1 \alpha'$.
 With the notations, we obtain the calculation for $h(\omega_R)$ as follows.

\begin{thm}\label{thm4.4}
Suppose that $a_1 \alpha$ is the smallest integer among $a_1\alpha$, $a_2 \beta$, $a_3 \gamma$, $a_2\beta'$, $a_3\gamma'$, and $a_1 \alpha'$. Then, $h(\omega_R)=\alpha \beta' (\gamma + \gamma')$.
\end{thm}

\begin{proof}
By \cite[Proposition 6.3]{hbs}, $\tr_R(\omega_R)=(t^{a_1\alpha}, t^{a_2\beta'}, t^{a_2 \beta}, t^{a_3 \gamma'}, t^{a_3 \gamma}, t^{a_1 \alpha'})$. It follows that $(t^{a_1\alpha}, t^{a_2\beta'})$ is a reduction of $\tr_R(\omega_R)$ since $a_1\alpha$ is the smallest integer in the value of elements in $\tr_R(\omega_R)$. Thus, by Theorem \ref{t26} and Proposition \ref{canon}, the canonical ideal $(t^{a_1\alpha}, t^{a_2\beta'})$ is a partial trace ideal of $\omega_R$.
Hence, we obtain that 
\begin{align*}
h(\omega_R) =& \ell_R(R/(t^{a_1\alpha}, t^{a_2\beta'})) = \ell_P\left(P/[I_2\left(\begin{smallmatrix}   
	X^{\alpha} & Y^{\beta} & Z^{\gamma} \\
	Y^{\beta'} & Z^{\gamma'} & X^{\alpha'}
\end{smallmatrix}\right) + (X^{\alpha}, Y^{\beta'})]\right)\\
=& \ell_P(k[X, Y, Z]/(X^{\alpha}, Y^{\beta'}, Z^{\gamma + \gamma'}-X^{\alpha'}Y^{\beta}))\\
=&\ell_P(k[X, Y, Z]/(X^{\alpha}, Y^{\beta'}, Z^{\gamma + \gamma'})),
\end{align*}
where the last equality follows from the minimality of $a_1 \alpha$ (recall that $a_1\alpha \le a_1\alpha'$). Thus, we have $h(\omega_R)=\alpha \beta' (\gamma + \gamma')$.
\end{proof}

\begin{ex}
Let $n\ge 1$ and let $R=k[t^{2n+1}, t^{2n+2}, t^{2n+3}]$ be a numerical semigroup ring over a field $k$. 
Then, $h(\omega_R) =n+1$. 
\end{ex}

\begin{proof}
It is straightforward to check that $\Ker \varphi$ contains $J:=I_2 \left(\begin{smallmatrix} X & Y & Z^n \\ Y & Z & X^{n+1}\end{smallmatrix}\right)$. We then have an exact sequence $0\to \Ker \varphi/J \to k[X, Y, Z]/J \to R \to 0$. Since $\varphi(X)=t^{a_1}$ is a non-zerodivisor of $R$, we get 
\[
0\to \Ker \varphi/(J+X\cdot \Ker \varphi) \to k[X, Y, Z]/(J+(X)) \to R/(t^{a_1}) \to 0.
\]
Noting that $\ell(k[X, Y, Z]/(J+(X))) = \ell(R/(t^{a_1})) =a_1$, we obtain that $\Ker \varphi/(J+X\Ker \varphi)=0$. By (graded) Nakayama's lemma, $\Ker \varphi=J$. Therefore, by Theorem \ref{thm4.4}, we get that $h(\omega_R) = n+1$. 
\end{proof}


\begin{thebibliography}{99}


\bibitem{AAM}
{\sc Ananthnarayan, H.; Avramov, Luchezar L.; Moore, W. Frank}. Connected sums of Gorenstein local rings. {\em J. Reine Angew. Math.} {\bf 667} (2012), 149--176.

\bibitem{Ba}
{\sc Bass, Hyman}. On the ubiquity of Gorenstein rings. {\em Math. Z.} {\bf 82} (1963), 8--28.



\bibitem{bh1}
{\sc Bruns, Winfried; Herzog, J\"{u}rgen}. Cohen-Macaulay rings. Cambridge Studies in Advanced Mathematics, {\bf 39}. Cambridge University Press, Cambridge, 1993. {\rm xii}+403 pp.

\bibitem{hw}
{\sc Celikbas, Olgur; Goto, Shiro; Takahashi, Ryo; Taniguchi, Naoki}. On the ideal case of a conjecture of Huneke and Wiegand. {\em Proc. Edinb. Math. Soc.} (2) {\bf 62} (2019), no. 3, 847--859.


\bibitem{DKT}  H. Dao, T. Kobayashi, R. Takahashi, Burch ideals and Burch rings, Algebra Number Theory 14 (8) (2020) 2121–2150.

\bibitem{dk}
{\sc Dey, Souvik; Kobayashi, Toshinori}. Vanishing of (co)homology of Burch and related submodules. {\em Illinois J. Math.} {\bf 67} (2023), no. 1, 101--151.

\bibitem{Ding}
{\sc Ding, Songqing}. A note on the index of Cohen-Macaulay local rings. {\em Comm. Algebra} {\bf 21} (1993), no. 1, 53--71.

\bibitem{teter}
{\sc Gasanova, Oleksandra; Herzog, J\"{u}rgen; Hibi, Takayuki; Moradi, Somayeh}. Rings of Teter type. {\em Nagoya Math. J.} {\bf 248} (2022), 1005--1033.

\bibitem{GIK}
{\sc Goto, Shiro; Isobe, Ryotaro; Kumashiro, Shinya}. Correspondence between trace ideals and birational extensions with application to the analysis of the Gorenstein property of rings. {\em J. Pure Appl. Algebra} {\bf 224} (2020), no. 2, 747--767.

\bibitem{GMP}
{\sc Goto, Shiro; Matsuoka, Naoyuki; Phuong, Tran Thi}. Almost Gorenstein rings. {\em J. Algebra} {\bf 379} (2013), 355--381.


\bibitem{GW}
{\sc Goto, Shiro; Watanabe, Keiichi}. On graded rings. I. {\em J. Math. Soc. Japan} {\bf 30} (1978), no. 2, 179--213.

\bibitem{Herzog}
{\sc Herzog, J\"{u}rgen}. Generators and relations of abelian semigroups and semigroup rings. {\em Manuscripta Math.} {\bf 3} (1970), 175--193.

\bibitem{hbs}
{\sc Herzog, J\"{u}rgen; Hibi, Takayuki; Stamate, Dumitru I.} The trace of the canonical module. {\em Israel J. Math.} {\bf 233} (2019), no. 1, 133--165.


\bibitem{HKS}
{\sc Herzog, J\"{u}rgen; Kumashiro, Shinya; Stamate, Dumitru I.} The tiny trace ideals of the canonical modules in Cohen-Macaulay rings of dimension one. {\em J. Algebra} {\bf 619} (2023), 626--642.

\bibitem{HKun}
{\sc Herzog, J\"{u}rgen; Kunz, Ernst}. Der kanonische Modul eines Cohen-Macaulay-Rings. Lecture Notes in Mathematics, {\bf 238}. Springer-Verlag, Berlin–New York, 1971. {\rm vi}+106 pp.



\bibitem{HS}
{\sc Huneke, Craig; Swanson, Irena}. Integral closure of ideals, rings, and modules. London Mathematical Society Lecture Note Series, {\bf 336}. Cambridge University Press, Cambridge, 2006. {\rm xiv}+431 pp.

\bibitem{HV}
{\sc Huneke, Craig; Vraciu, Adela}. Rings that are almost Gorenstein. {\em Pacific J. Math.} {\bf 225} (2006), no. 1, 85--102.

\bibitem{Ko}
{\sc Kobayashi, Toshinori}. Local rings with self-dual maximal ideal. {\em Illinois J. Math.} {\bf 64} (2020), no. 3, 349--373.


\bibitem{Lin}
{\sc Lindo, Haydee}. Trace ideals and centers of endomorphism rings of modules over commutative rings. {\em J. Algebra} {\bf 482} (2017), 102--130.


\bibitem{sm}
{\sc Maitra, Sarasij}. Partial trace ideals and Berger's conjecture. {\em J. Algebra} {\bf 598} (2022), 1--23.

\bibitem{Mai2}
{\sc Maitra, Sarasij}. Partial trace ideals, torsion and canonical module. {\em J. Algebra} {\bf 652} (2024), 1--19.


\bibitem{matsu}
{\sc Matsumura, Hideyuki}. Commutative ring theory. Translated from the Japanese by M. Reid. Cambridge Studies in Advanced Mathematics, {\bf 8}. Cambridge University Press, Cambridge, 1986. {\rm xiv}+320 pp.

\bibitem{Vas}
{\sc Vasconcelos, Wolmer V}. Computing the integral closure of an affine domain. {\em Proc. Amer. Math. Soc.} {\bf 113} (1991), no. 3, 633--638.


\end{thebibliography}
\end{document}